\documentclass[12pt]{article}
\usepackage{amsmath,amssymb,amsfonts}
\usepackage{graphicx,psfrag,epsf}
\usepackage{enumerate}
\usepackage{url} 

\newtheorem{theorem}{Theorem}
\newtheorem{proposition}{Proposition}
\newtheorem{definition}{Definition}

\newtheorem{application}{Application}

\newenvironment{proof}{\emph{Proof.}} {\quad \hfill $\blacksquare$ \newline}

\usepackage{color}
\usepackage{soul} 

\begin{document}

\def\spacingset#1{\renewcommand{\baselinestretch}%
{#1}\small\normalsize} \spacingset{1}



  \date{}
  \title{\bf Applications of the Quantile-Based Probabilistic
Mean Value Theorem to Distorted Distributions}
  \author{Antonio Di Crescenzo\\
   Dipartimento di Matematica, Università di Salerno,\\ Fisciano, ITALY\\\texttt{adicrescenzo@unisa.it}\\
    and\\
    Barbara Martinucci\\
    Dipartimento di Matematica, Università di Salerno,\\ Fisciano, ITALY\\ \texttt{bmartinucci@unisa.it}\\
    and \\
    Julio Mulero\\
    Departamento de Matem\'aticas, Universidad de
Alicante, \\ Alicante, SPAIN \\ \texttt{julio.mulero@ua.es}}
  \maketitle
  \newpage

\newpage
{\LARGE  \begin{center}
    \bf Applications of the Quantile-Based Probabilistic
Mean Value Theorem to Distorted Distributions
\end{center}}

\bigskip
\begin{abstract}
 Distorted distributions were introduced in the context of
actuarial science for several variety of insurance problems. In this paper
we consider the quantile-based probabilistic mean value theorem given in
Di Crescenzo et al. [4] and provide some applications based on distorted
random variables. Specifically, we consider the cases when the underlying
random variables satisfy the proportional hazard rate model and
the proportional reversed hazard rate model. A setting based on random
variables having the ‘new better than used’ property is also analyzed.
\end{abstract}

\noindent%

\bigskip
\spacingset{1.45} 
\section{Introduction and background} 
A probabilistic generalization of the Taylor’s theorem was proposed and studied
by Massey and Whitt [8] and Lin [7], showing that for a nonnegative random
variable $X$ and a suitable function $f$ one has
\[
E[f(t+X)]=\sum_{k=0}^{n-1}E[X^k]\frac{f^{(k)}(t)}{k!}+E[f^{(n)}(t+X_e)]\frac{f^{(n)}(t)}{n!}, \,\,\,t>0,
\]
where $x_e$ is a random variable possessing the equilibrium distribution of $X$.
This result was employed by Di Crescenzo [3] in order to obtain the following
probabilistic version of the well-known mean value theorem:
\[
E[g(Y)]-E[g(X)] = E[g(Z)][E(Y)-E(X)],
\]
where $X$ and $Y$ are nonnegative random variables such that $X \leq_{st} Y$ , i.e.,
$P(X>x) \leq P(Y>x)$, for all $x \geq 0$, and $E(X) < E(Y )$. Moreover, the random
variable $Z$ is a generalization of $X_e$ and has density function
\[
\frac{P(Z\in dx)}{dx}=\frac{P(Y>x)-P(X>x)}{E(Y)-E(X)}.
\]
Various related results have been exploited recently, such as the fractional
probabilistic Taylor's and mean value theorems (see Di Crescenzo and Meoli
[5]), and a quantile-based version of the probabilistic mean value theorem (see
Di Crescenzo et al. [4]). The latter involves a distribution that generalizes the
Lorenz curve, and allows the construction of new distributions with support
$(0, 1)$. Specifically, for any random variable $X$, let $F(x) = P(X \leq x)$, $x \in R$,
denote the distribution function, and let
\begin{equation}
Q(u)=\inf\{x\in\mathbb R: F(x)\geq u\}, \,\,\,0<u<1
\end{equation}
be the quantile function, with $Q(0) := lim_{u\rightarrow 0} Q(u)$ and $Q(1) := lim_{u\rightarrow 1} Q(u)$.

\begin{definition}
Let $\mathcal D$ be the family of all absolutely continuous random variables
having finite nonzero mean, and such that the quantile function (1) satisfies
$Q(0) = 0$ and is differentiable, in order that the following quantile density
function exists:
\[
q(u)=Q'(u),\,\,\, 0<u<1.
\]
\end{definition}

We remark that if $X \in\mathcal D$ then $F[Q(u)] = u$, $0 < u < 1$. According to Di
Crescenzo et al. [4], if $X \in \mathcal D$, let $X^L$ denote an absolutely continuous random
variable taking values in $(0, 1)$ with distribution function
\begin{equation}
L(p)=\frac{1}{E[X]}\int_0^pQ(u)du,\,\,\,0\leq p\leq 1,
\end{equation}
and density
\[
f_{X^L}(u)=\frac{Q(u)}{E[X]},\,\,\, 0<u<1.
\]
Note that the function given in (2) is also known as the Lorenz curve of $X$,
and deserves large interest in mathematical finance for the representation of the
distribution of income or of wealth. Then, for $X \in\mathcal D$, if $g : (0, 1) \rightarrow \mathbb R$ is $n$-times
differentiable and $g^{(n)} \cdot Q$ is integrable on $(0, 1)$ for any $n \geq 1$, then
\begin{multline*}
E[\{g(1)-g(U)\}q(U)]=\sum_{k=1}^{n-1}E[g^{(k)}(U)(1-U)^kq(U)]\\+\frac{1}{(n-1)!}E[g^{(n)}(X^L)(1-X^L)^{n-1})]E[X],
\end{multline*}
where $U$ is uniformly distributed in $(0, 1)$. Furthermore, the following result can
be viewed as a quantile-based analogue of the probabilistic mean value theorem
(cf. Theorem 3 in Di Crescenzo et al. [4]). In particular, given $X, Y \in\mathcal D$ such
that $X \leq_{st} Y$ and a differentiable 
function $g: (0, 1) \rightarrow \mathbb R$ 
with $g'\cdot Q_X$ and $g'\cdot Q_Y$
integrable on $(0, 1)$, then
\begin{equation}
E[\{g(1)-g(U)\}\{q_Y (U)-q_X(U)\}] = E[g'(Z^L)
\{E[Y ]-E[X]\},
\end{equation}
where $U$ is uniformly distributed in $[0, 1]$ and $q_X$ and $q_Y$ are the quantile densities
of $X$ and $Y$, respectively. Moreover, $Z^L$ denotes a random variable having
distribution function
\[
L_{X,Y}(p)=\frac{1}{E[Y]-E[X]}\int_0^p[Q_Y(u)-Q_X(u)]du,\,\,\,0\leq p\leq 1,
\]
which is a suitable extension of (2). Note also that, under the previous assumptions,
$E[g'(Z^L)]$ in (3) is finite.

Stimulated by the above mentioned results, in this paper we construct new
relationships involving distorted random variables that deserve interest in utility
theory and can be applied for assessing stochastic dominance among risks (we
refer, for instance, to the recent papers by Balbás et al. [1] and Sordo et al.
[12,14]).

Let $\Gamma$ be the set of continuous, nondecreasing and piecewise differentiable
functions $h : [0, 1] \rightarrow [0, 1]$ such that $h(0) = 0$ and $h(1) = 1$. These functions are
called distortion functions. See Sordo and Suárez-Llorens [13] for applications of
distortion functions to classes of variability measures, and Gupta et al. [6] for
the use of distortion functions for the analysis of random lifetimes of coherent
systems. We denote by $\bar F(x) = 1-F(x)$, $x\in\mathbb R$, the survival function of $X$.

\begin{definition}
For each distortion function $h\in\Gamma$, and any survival function $F(x)$, the position
\[
\bar F_h(x) = h(\bar F(x)),\,\,\, x\in\mathbb R,
\]
defines a survival function associated to a new random variable $X_h$, which is
called the distorted random variable induced by $h$.
\end{definition}

Distorted distributions were introduced by Denneberg [2] and Wang [15,16]
in the context of actuarial science for several variety of insurance problems. One
of the most important applications is in the rank dependent expected utility
model (see Quiggin [10], Yaari [17], Schmeidler [11]).

It is easy to see that given a random variable $X$ and a distortion function
$h\in\Gamma$, the distorted random variable induced by $h$, say $X_h$, has quantile function given by
\[
Q_h(u):=Q(1-h^{-1}(1-u)),\,\,\, 0<u<1,
\]
where $Q$ is the quantile function of $X$.

\begin{proposition}
Let $X$ be a nonnegative random variable, and $h, l \in\Gamma$ two distortion functions. Then,
\[
X_h \leq_{st}X_l \text{ if, and only if, } h(x) \leq l(x),\,\,\, 0 < x < 1.
\]
\end{proposition}
\begin{proof}
The proof immediately follows noting that $X_h \leq_{st} X_l$ holds if, and only
if, $h(F(x)) \leq l(F(x))$ or, equivalently, $h(x) \leq l(x)$, for all $0 < x < 1$.
\end{proof}

\section{Results Based on Distorted Random Variables}

In this section we provide some applications of (3) based on the comparisons of
distorted random variables.

\begin{theorem}
Let $X \in\mathcal D$ be a nonnegative random variable with quantile function
$Q$ and quantile density $q$. Let h and l be two distortion functions such that
$h(x) \leq l(x)$, for all $0 < x < 1$, one has $E[X_l] < E[X_h] < +\infty$. Then, for a
random variable $U$ uniformly distributed in $(0, 1)$ we have that
\begin{multline*}
E\left[\left(\frac{q(1-l^{-1}(1-U))}{l'(l^{-1}(1-U))}-\frac{q(1-h^{-1}(1-U))}{h'(h^{-1}(1-U))}\right)(g(1)-g(U))\right]\\=E[g'(Z^L)](E[X_l]-E[X_h]),
\end{multline*}
where 
\[
E[X_h]=\int_0^\infty h(\bar F(t))dt,
\]
and $Z^L$ is the random variable with density function
\begin{equation}
f_{Z^L}(x)=\frac{Q(1-l^{-1}(1-x))-Q(1-h^{-1}(1-x))}{E[X_l]-E[X_h]},\,\,\,0<x<1.
\end{equation}
\end{theorem}

\begin{proof}
Denoting as $q_h$ and $q_l$ the quantile densities corresponding to $Q_h$ and $Q_l$,
respectively. Then, for $0 < u < 1$, one has
\[
q_h(u)=\frac{q(1-h^{-1}(1-u))}{h'(h^{-1}(1-u)}\text{ and }q_l(u)=\frac{q(1-l^{-1}(1-u))}{l'(l^{-1}(1-u)}.
\]
Hence, the thesis follows from (3).
\end{proof}

There exist several types of distortion functions that leads to special cases of
interest. For instance, if $h(t) = t^\alpha$, then $X_h$ and $X_l$ correspond to the proportional
hazard rate model (see, for instance, Balbás et al. [1] and Navarro et al.
[9]). This suggests our first application.

\begin{application}
Let us consider the distortion functions $h(t) = t^\alpha$ and $l(t) = t^\beta$
for $0 < \beta < \alpha$ and $0 < t < 1$, so that $h(t) \leq l(t)$ for $0 < t < 1$. We can
consider the distorted random variables induced by $h$ and $l$, say $X_h$ and $X_l$,
respectively. Hence, due to Proposition 1, we have $X_h \leq_{st} X_l$. It is easy to see
that the survival functions of $X_h$ and $X_l$ are
\[
\bar F_h(x)=(\bar F(x))^\alpha\text{ and }\bar F_l(x)=(\bar F(x))^\beta,\,\,\, x>0,
\]
respectively. With straightforward calculations, we can obtain the corresponding
quantile functions
\[
Q_h(u)=Q(1-(1-u)^{1/\alpha})\text{ and }\bar Q_l(u)=Q(1-(1-u)^{1/\beta}),
\]
and the quantile densities
\[
q_h(u)=\frac{1}{\alpha} q(1-(1-u)^{1/\alpha})(1-u)^{\frac{1}{\alpha}-1}\text{ and }q_l(u)=\frac{1}{\beta} q(1-(1-u)^{1/\beta})(1-u)^{\frac{1}{\beta}-1},
\]
for $0 < u < 1$, where $Q$ and $q$ are respectively the quantile and quantile density
function of $X$. Then, from Theorem 1, we have
\begin{multline*}
E\left[\left(\frac{q(1-(1-U)^{1/\beta})}{(1-U)^{1-\frac{1}{\beta}}}-\frac{q(1-(1-U)^{1/\alpha})}{(1-U)^{1-\frac{1}{\alpha}}}\right)(g(1)-g(U))\right]\\=E[g'(Z^L)](E[X_l]-E[X_h]),
\end{multline*}
where the density of $Z^L$, given in (4), becomes
\[
f_{Z^L}(x)=\frac{Q(1-(1-x)^{1/\alpha})-Q(1-(1-x)^{1/\beta})}{E[X_l]-E[X_h]},\,\,\,0<x<1.
\]
with $E[X_l]=\int_0^\infty(\bar F(x))^\alpha dx$ and $E[X_h]=\int_0^\infty(\bar F(x))^\beta dx$. For instance, if $X$ is uniformly distributed in $(0,1)$, then $!(u)=u$m abd $q(u)=1$, $0<u<1$, so that $E[X_l]=(\alpha+1)^{-1}$ and $E[X_h]=(\beta+1)^{-1}$. Therefore,
\begin{equation*}
E\left[\left(\frac{(1-U)^{\frac{1}{\beta}-1}}{\beta}-\frac{(1-U)^{\frac{1}{\alpha}-1}}{\alpha}\right)(g(1)-g(U))\right]=E[g'(Z^L)]\frac{\alpha-\beta}{(\alpha+1)(\beta+1)},
\end{equation*}
where, for $0\beta<\alpha$,
\[
f_{Z^L}(x)=\frac{(\alpha+1)(\beta+1)}{\alpha-\beta}((1-x)^{1/\alpha}-(1-x)^{1/\beta}),\,\,\, 0<x<1.
\]
\end{application}

Next we consider $h(t) = 1-(1-t)^n$, $0 < t < 1$, for some positive integer $n$.
Note that in this case $E[X_h] = E[\max\{X_1, \dots, X_n\}]$, where $X_1, \dots, X_n$ are i.i.d.
random variables.

\begin{application}
Let us consider the distortion functions $h(t) = 1-(1-t)^n$ and
$l(t) = 1-(1-t)^m$, $0 < t < 1$, for integers $1 \leq n \leq m$. It is not hard to see that
these distortion functions refer to the proportional reversed hazard rate model.
Hence, since $h(t) \leq l(t)$, $0 < t < 1$, the corresponding distorted random variables
satisfy $X_h \leq_{st} X_l$ due to Proposition 1. The survival and quantile functions of
$X_h$ and $X_l$ are respectively
\[
\bar F_h(x)=1-F^n(x)\text{ and }\bar F_l(x)=1-F^m(x),\,\,\, x>0,
\]
\[
Q_h(x)=Q(1-u^{1/n})\text{ and }Q_l(x)=Q(1-u^{1/m}),\,\,\, 0<u<1,
\]
where $F(x)$ and $Q(u)$ are the distribution and the quantile function of the i.i.d.
random variables $X_1, \dots, X_n$, respectively. If $M_k:=\max\{X_1, \dots, X_k\}$, for $k\geq 1$, then from Theorem 1 we have
\begin{equation}
E\left[\left(\frac{q(U^{\frac{1}{m}})}{mU^{1-\frac{1}{m}}}-\frac{q(U^{\frac{1}{n}})}{nU^{1-\frac{1}{n}}}\right)(g(1)-g(U))\right]=E[g'(Z^L)](E[M_m]-E[M_n]),
\end{equation}
where
\[
f_{Z^L]}(x)=\frac{Q(x^{\frac{1}{m}})-Q(x^{\frac{1}{n}})}{E[M_m]-E[M_n]},\,\,\, 0<x<1.
\]
For instance, if $X$ is exponentially distributed with parameter $\lambda > 0$, it is known
that, for each integer $n \geq 1$,
\[
E[M_n]=\frac{1}{\lambda}\int_0^1\frac{1-t^n}{1-t}dt=\frac{1}{\lambda}H_n,
\]
where $H_n:=\sum_{k=1}^n\frac{1}{k}$ is the $n$th harmonic number. Hence, for $1\leq n\leq m$, from (5) and given that $q(u)=[\lambda(1-u)]^{-1}$, for all $0<u<1$, we have
\begin{equation*}
E\left[\left(\frac{1}{m}\frac{1}{U^{-1/m}-1}-\frac{1}{n}\frac{1}{U^{-1/n}-1}\right)(g(1)-g(U))\right]=E[g'(Z^L)](H_m-H_n),
\end{equation*}
where
\[
f_{Z^L}(x)=\frac{1}{H_m-H_n}\log\frac{1-x^{1/n}}{1-x^{1/m}},\,\,\, 0<x<1.
\]
The following theorem involves a random variable having the NBU property.
We recall that a nonnegative random variable $X$ is said NBU if its survival
function $\bar F$ satisfies $\bar F(s)\bar F(t) \geq \bar F(s + t$), for all $s \geq 0$ and $t \geq 0$.

\begin{theorem}
Let $X$ be a nonnegative random variable with quantile density $q$,
and having the NBU property, and let $h$ be a distortion function. Then, for a
random variable $U$ uniformly distributed in $(0, 1)$, and for all $t > 0$, one has
\begin{multline*}
E\left[\left(\frac{q(1-h^{-1}(1-U))}{h'(h^{-1}(1-U))}-\bar F(t)\frac{q(1-h^{-1}(1-U)\bar F(t))}{h'(h^{-1}(1-U))}\right)(g(1)-g(U))\right]\\=E[g'(Z^L)](E[X_h]-E[(X_t)_h]),
\end{multline*}
where $E[X_h]=\int_0^\infty h(\bar F(t))dt$ and, given $t>0$, $(X_t)_h$ and $Z^L$ are respectively the random variables having survival and density functions given by
\[
\bar F_{(X_t)_h}(x)=h\left(\frac{\bar F(x+t)}{\bar F(t)}\right),\,\,\, x\geq 0,
\]
\[
f_{Z^L}(x)=\frac{Q(1-h^{-1}(1-x))-Q(1-h^{-1}(1-x)\bar F(t))+t}{E[X_h]-E[(X_t)_h]},\,\,\, 0<x<1.
\]
\end{theorem}
\begin{proof}
Denoting as $q_h$ and $q_{t,h}$ the quantile densities of $X_h$ and $(X_t)_h$, respectively,
for $0 < u < 1$ we have
\[
q_h(u)=\frac{q(1-h^{-1}(1-u))}{h'(h^{-1}(1-u))}\text{ and }q_{t,h}(u)=\bar F(t)\frac{q(1-h^{-1}(1-u)\bar F(t))}{h'(h^{-1}(1-u))}.
\]
The thesis then follows from Theorem 1.
\end{proof}
\end{application}

\begin{application}
Let $X$ be an uniformly distributed random variable in $(0, 1)$.
It is well known that $X$ is NBU. The survival and quantile function of $X_h$ are
given respectively by
\[
\bar F_h(x)=h(1-x),\,\,\, 0<x<1,\text{ and }Q_h(u)=1-h^{-1}(1-u),\,\,\, 0<u<1.
\]
Similarly, for the distorted random variable $(X_t)_h$ given $t \in (0, 1)$, we have
\[
\bar F_{(X_t)_h}(x)=h\left(\frac{1-(t+x)}{1-t}\right),\,\,\, 0<x<1-t,
\]
\[
Q_{(X_t)_h}(x)=1-(2-t)h^{-1}(1-u),\,\,\, 0<u<1,
\]
The corresponding quantile densities of $X_h$ and $(X_t)_h$ are
\[
q_h(u)=\frac{1}{h'(h^{-1}(1-u))}\text{ and }q_{t,h}(u)=\frac{1-t}{h'(h^{-1}(1-u))} ,\,\,\, 0<u<1,
\]
From Theorem 2, we have
\begin{multline}
E\left[\left(\frac{1}{h'(h^{-1}(1-U))}-\frac{1-t}{h'(h^{-1}(1-U))}\right)(g(1)-g(U))\right]\\=E[g'(Z^L)](E[X_h]-E[(X_t)_h]),
\end{multline}
where $Z^L$ is the random variable having density
\[
f_{Z^L}(x)=\frac{(1-t)h^{-1}(1-x)-h^{-1}(1-x)}{E[X_h]-E[(X_t)_h]},\,\,\, 0<x<1,
\]
Let us now consider 
\[
h(t) = \min\left\{\frac{t}{1-p},1\right\},
\]
 for a fixed $p \in (0, 1)$. It can be seen
that $h$ is a proper distortion function, and that $E[X_h] = E[X|X >Q(p)]$ for any
random variable $X$ having quantile function $Q(p)$. Specifically, if $X$ is uniformly
distributed in $(0, 1)$, we have
\[
E[X|X>Q(p)]=\frac{1+p}{2}\text{ and }E[X_t|X_t>Q_t(p)]=\frac{(1-t)(1+p)}{2}.
\]
Hence, since $h^{-1}(u)=(1-p)u$ and $h'(u)=\frac{1}{1-p}$, $0<u<1$, (6) gives
\[
(1-p)E[g(1)-g(U)]=E[g'(Z^L)]\frac{1+p}{2},\,\,\, 0<p<1,
\]
where the density of $Z^L$ is
\[
f_{Z^L}(x)=2\frac{1+(1-p)(1-x)}{1+p},\,\,\, 0<x<1.
\]
\end{application}

\section{Conclusions}

The quantile-based probabilistic mean value theorem proposed in [4] has been
shown to be useful (i) to construct new probability densities with support $(0, 1)$
starting from suitable pairs of stochastically ordered random variables, and (ii)
to obtain equalities involving uniform-$(0, 1)$ distributions and quantile functions.
On this ground, further applications have been provided in the present paper
based on distorted distributions, with special care to the cases when the underlying
random variables satisfy the proportional hazard rate model, the proportional
reversed hazard rate model, and the 'new better than used' property.

\end{document}